\documentclass[12]{amsart}
\usepackage{amssymb,amsmath,amsfonts,amsthm}
\usepackage[dvips]{graphicx}
\usepackage{psfrag}

\theoremstyle{plain}
\newtheorem{main}{Theorem}

\newtheorem{theorem}{Theorem}[section]
\newtheorem{lemma}[theorem]{Lemma}
\newtheorem{proposition}[theorem]{Proposition}
\newtheorem{corollary}[theorem]{Corollary}
\theoremstyle{remark}
\newtheorem{remark}[theorem]{Remark}

\newcommand{\C}{\operatorname{C}}
\newcommand{\Gibb}{\operatorname{Gibb}}
\newcommand{\MM}{\operatorname{MM}}
\newcommand{\Basin}{\operatorname{Basin}}

           \def\ea{\end{array}}
          \def\ec{\end{center}}
     \def\ed{\end{description}}
        \def\ee{\end{equation}}
       \def\eea{\end{eqnarray}}
     \def\eeaa{\end{eqnarray*}}
 \def\et{\end{thebibliography}}

\def\Orb{{\rm Orb}}
\def\Diff{{\rm Diff}}
\def\Cl{{\rm Cl}}

\def\MM{\operatorname{MM}}
\def\supp{\operatorname{supp}}
\def\PH{\operatorname{PH}(M)}
\def\SPH1{\operatorname{SPH}_1(M)}

\def\cU{{\mathcal U}}

\def\cH{{\mathcal H}}

\def\cF{{\mathcal F}}

\def\cP{{\mathcal P}}

\def\cS{{\mathcal S}}

\def\cW{{\mathcal W}}

\def\TT{{\mathbb T}}

\def\tB{\tilde{B}}
\def\txi{\tilde{\xi}}
\def\tcH{\tilde{\mathcal{H}}}
\def\hnu{\hat{\mu}}

\title[Diffeomorphisms of circle fiber bundles]{Maximal entropy measures of diffeomorphisms of circle fiber bundles}
\author{Ra\'ul Ures, Marcelo Viana and Jiagang Yang}
\date{\today}
\thanks{M.V. and J.Y. were partially supported by CNPq, FAPERJ, and PRONEX. R. U. was partially supported by Southern University of Science and Technology.
We acknowledge support from the Fondation Louis D--Institut de France (project coordinated by M. Viana).}

\address{Department of Mathematics, Southern University of Science and Technology, 1088 Xueyuan Rd., Xili, Nanshan District, Shenzhen, Guangdong, China 518055} \email{ures@sustc.edu.cn}

\address{IMPA, Est. D. Castorina 110, 22460-320 Rio de Janeiro, Brazil}
\email{viana\@@impa.br}

\address{Departamento de Geometria, Instituto de Matem\'atica e Estat\'\i stica, Universidade Federal Fluminense, Niter\'oi, Brazil}
\email{yangjg\@@impa.br}

\begin{document}

\begin{abstract}
We characterize the maximal entropy measures of partially hyperbolic $C^2$ diffeomorphisms whose center foliations
form circle bundles, by means of suitable finite sets of saddle points, that we call skeletons.

In the special case of 3-dimensional nilmanifolds other than the torus, this entails the following dichotomy:
either the diffeomorphism is a rotation extension of an Anosov diffeomorphism -- in which case there is a unique
maximal measure, with full support and zero center Lyapunov exponents -- or there exist exactly two ergodic maximal
measures, both hyperbolic and whose center Lyapunov exponents have opposite signs.
Moreover, the set of maximal measures varies continuously with the diffeomorphism.
\end{abstract}

\maketitle

\section{Introduction}
The metric entropy describes the complexity of a dynamical system relative to an invariant probability measure.
For diffeomorphisms of a compact Riemannian manifold, the variational principle (see \cite{LW}) states that the
supremum of the metric entropy over all invariant probability measures coincides with the topological entropy
of the system.

We call \emph{maximal (entropy) measure} any invariant probability measure whose metric entropy coincides with the topological entropy.
Such measures reflect the complexity level of the whole system, and constitute a classical topic in ergodic theory:
Are there maximal measures? How many ergodic maximal measures does the system have, and where are they supported?
How do they vary with the dynamical system?

It is a classical fact, see for instance \cite{B}, that every transitive hyperbolic set admits a unique maximal measure.
Recently, new approaches were developed for the study of maximal measures of non-uniform hyperbolic maps and
partially hyperbolic diffeomorphisms, including $C^\infty$ interval maps \cite{Bu},
$C^{1+\alpha}$ surface diffeomorphisms \cite{BCS,S} and Derived from Anosov diffeomorphisms \cite{BFSV, U, VY,YY}.
For some recent progress in the more general setting of equilibrium states, see \cite{CPZ} and references therein.

In the present paper, we study $C^{2}$ partially hyperbolic diffeomorphisms with 1-dimensional center and
whose center foliation forms a circle bundle.
It was shown in \cite{HHTU} that if the diffeomorphism is accessible and not rotation type then it admits finitely
many ergodic maximal measures.
Those conditions are satisfied on an open and dense subset of partially hyperbolic diffeomorphisms.

Here we use a combinatorial object -- a finite set of hyperbolic saddles, that we call \emph{skeleton} --
to describe all the maximal measures. More precisely, by observing those saddles, we can count the number of
ergodic maximal measures, locate their supports, and explain how they vary with respect to the diffeomorphism.
Our main detailed results will be stated in Section~\ref{s.general_results}.
Right now, let us mention some applications to diffeomorphisms on 3-dimensional \emph{nilmanifolds}, that is,
manifolds $M$ that are circle bundles over the torus $\TT^2$ (the case $M=\TT^3$ is excluded).

\begin{main}\label{main.Nil}

Let $f$ be a $C^2$ partially hyperbolic diffeomorphism on a 3-dimensional nilmanifold $M \neq \TT^3$.
Then

\begin{itemize}

\item[(a)] either $f$ has a unique maximal measure, in which case $f$ is conjugate to a rotation extension of an Anosov diffeomorphism,
and the maximal measure is supported on the whole $M$ and has vanishing center exponent;

\item[(b)] or $f$ has exactly two ergodic maximal measures $\mu^+,\mu^-$, with positive and negative center Lyapunov exponents, respectively.
\end{itemize}

\end{main}

Diffeomorphisms as in (a) are called \emph{rotation type}.
These are clearly rigid systems: for instance, they cannot admit hyperbolic periodic orbits; moreover, the complement is a $C^1$ open and
$C^r$ ($r\geq 1$) dense subset of the set $\PH$ of all $C^2$ partially hyperbolic diffeomorphisms on $M$.

We also get the following global result, asserting that the ergodic maximal measures vary continuously
with the diffeomorphism. To state this precisely, let $\cP(M)$ be the space of probability measures on $M$
endowed with the weak$^*$ topology. Let $\PH$ be endowed with the $C^1$ topology.

\begin{main}\label{main.Nilcontinuation}
Let $M$ be a 3-dimensional nilmanifold different from $T^3$. Then there are continuous functions,
$$
\Gamma^+, \Gamma^- \colon \PH \to \cP(M),
$$
such that for each $f\in\PH$, $\Gamma^+(f)$ is the ergodic maximal measure of $f$ with non-negative
center exponent, $\Gamma^-(f)$ is the ergodic maximal measure of $f$ with non-positive center exponent,
and $\Gamma^+(f)=\Gamma^-(f)$ if and only if either of them has vanishing center exponent.
\end{main}

Theorems~\ref{main.Nil} and~\ref{main.Nilcontinuation} will be deduced from much more detailed results that
we state in the next section.

\section{Statement of results}\label{s.general_results}

\subsection{Objects}

A diffeomorphism $f:M \to M$ is \emph{partially hyperbolic}, if the tangent bundle splits into three invariant subbundles,
$TM=E^s\oplus E^c\oplus E^u$, such that $E^s$ is uniformly contracting, $E^u$ is uniformly expanding, and $E^c$ has
intermediate behavior. By this we mean that there exists some Riemannian metric on $M$ such that
$$
\|df\mid_{E^s}\|<1, \ \|df^{-1}\mid_{E^u}\|<1 , \text{ and } \|df_x(e^s)\|<\|df_x(e^c)\|<\|df_x(e^u)\|
$$
for all $e^\sigma\in E^\sigma_x$, $||e^\sigma||=1$, $\sigma\in\{s,c,u\}$, and all $x\in M$.

A partially hyperbolic diffeomorphism $f$ is \emph{accessible} if any two points $x,y$ can be joined
by a curve formed by finitely many arcs which are tangent to either the strong stable subbundle $E^s$ or
the strong unstable bundle $E^u$.
A partially hyperbolic diffeomorphism $f$ is \emph{dynamically coherent} if for $i=cs, cu$,
there is an invariant foliation $\cF^{i}$ tangent to the bundle $E^i$, where $E^{cs}=E^c\oplus E^s$ and $E^{cu}=E^u\oplus E^c$.

Recall that accessibility is a $C^1$ open and $\C^r$ ($r\geq 1$) dense property for the partially
hyperbolic diffeomorphisms with 1-dimensional center direction (\cite{BHHTU,D}).
Moreover, when $f$ is dynamically coherent $\cF^c=\cF^{cu}\bigcap \cF^{cs}$ is an invariant center foliation of $f$,
that is, tangent to the center subbundle $E^c$.

Denote by $\SPH1$ the set of $C^2$ partially hyperbolic, accessible, dynamically coherent
diffeomorphisms with 1-dimensional center direction for which the center foliation $\mathcal{F}^c$
forms a circle bundle.
For $f\in\SPH1$, let $f_c$ denote the map induced by $f$ on the quotient space $M_c=M/\mathcal{F}^c$.
Then $f_c$ is a topological Anosov homeomorphism (a globally hyperbolic homeomorphism,
in the sense of \cite[Section~1.3]{V} or \cite[Section 2.2]{VY}).
We use $\cF^i$, $i\in\{s,c,u\}$ to denote the invariant foliations of $f$, and $\cW^i$, $i\in\{s,u\}$
to denote the stable and unstables foliations of $f_c$, respectively. We further assume that
\begin{equation}\label{eq.transtivebase}
\text{ $M_c$ is a torus.}
\end{equation}

\begin{remark}\label{r.baseconjugate}
By a result of Hiraide~\cite{Hi}, $f_c$ is conjugate to a linear Anosov torus diffeomorphism, for any $f\in \SPH1$.
In particular, $f_c$ is a transitive homeomorphism, and admits a unique probability measure of maximal entropy $\nu$.
\end{remark}

It follows from the classification results in \cite{H,HP} that if $M$ is a 3-dimensional
nilmanifold $M$ other than $\TT^3$ then $\SPH1$ contains every $C^2$ partially hyperbolic diffeomorphism of $M$:

\begin{proposition}\cite[Propositions~1.9 and~6.4]{HP}\label{p.nilmanifold}
If $f:M\to M$ is a $C^2$ partially hyperbolic diffeomorphism on a 3-dimensional nilmanifold $M$ other than $\TT^3$
then it admits a unique center foliation, tangent to the center bundle $E^c$, and which forms a circle bundle.
Moreover, $f$ is accessible, and it has a unique compact, invariant, $u$-saturated (respectively, $s$-saturated) minimal subset.
\end{proposition}

From now on, we restrict ourselves to the diffeomorphisms in $\SPH1$.

\subsection{Main results}

Let $f:M\to M$ be a partially hyperbolic diffeomorphism. A finite set $S=\{p_1,\cdots,p_k\}$ is a \emph{skeleton} of $f$ if
\begin{itemize}
\item[(a)] each $p_i$ ($1\leq i \leq k$) is a hyperbolic saddle of $f$ with stable index $\dim(E^{cs})$;
\item[(b)] $\cF^u(x)$ intersects $\cup_{1\leq i \leq k} W^s(\Orb(p_i))$ transversely at some point, for every $x\in M$;
\item[(c)] $W^u(p_i)\cap W^s(\Orb(p_j))=\emptyset$ for $1\leq i \neq j \leq k$.
\end{itemize}
A finite set satisfying the conditions (a) and (b) is called a \emph{pre-skeleton}.
A pre-skeleton is \emph{minimal} if no strict subset is a pre-skeleton.
Skeletons not always exist. The following facts about skeletons and pre-skeletons can be found in \cite{DVY}:

\begin{proposition}\label{p.skeleton}~
\begin{itemize}
\item[(a)] If both $S_1$ and $S_2$  are skeletons of $f$ then $\# S_1 = \# S_2$.
\item[(b)] If $\{p_1,\cdots, p_k\}$ is a pre-skeleton of $f$ then $\{p_1(g),\cdots, p_k(g)\}$ is a pre-skeleton of any $g$
sufficiently close to $f$, where $p_i(g)$ denotes the hyperbolic continuation of $p_i$.
\item[(c)] If $\{p_1,\cdots,p_k\}$ is a skeleton of $f$ then, for any $g$ sufficiently close to $f$,
$\{p_1(g),\cdots, p_k(g)\}$ is a skeleton of $g$ if and only if no heteroclinic intersection was created relating
$p_i(g)$ and $p_j(g)$, for any $1\leq i \neq j \leq k$.
\item[(d)] Every pre-skeleton contains a minimal pre-skeleton, and a pre-skeleton is minimal if and only if it is a skeleton.
\end{itemize}
\end{proposition}

Part (b) ensures that the subset of diffeomorphisms with a pre-skeleton is open.

We will deduce Theorems~\ref{main.Nil} and \ref{main.Nilcontinuation} from the two theorems that follow,
by means of Proposition~\ref{p.nilmanifold}. Recall that, given a foliation $\cF$, we say that a compact subset of $M$
is \emph{$\cF$-saturated} if it consists of entire leaves.
Following \cite{BV}, we call an $\cF$-saturated set an \emph{$\cF$-minimal component} if each leaf contained in it is dense.
We use $u$-saturated and $s$-saturated as synonyms to $\cF^u$-saturated and $\cF^s$-saturated, respectively.

\begin{main}\label{main.skeleton}
Let $f\in \SPH1$. If $f$ is rotation type then it has a unique maximal measure, and this measure has vanishing center exponent.
Otherwise, $f$ and $f^{-1}$ have skeletons $S(f)=\{p_1,\cdots,p_k\}$ and $S(f^{-1})=\{q_1,\cdots,q_l\}$, respectively,
and $f$ has exactly $k+l$ ergodic maximal measures:
\begin{itemize}
\item[(a)] $k$ ergodic maximal measures $\mu^-_i$ with negative center exponents; each support $\supp\mu^-_i$
coincides with $\Cl(\cF^u(\Orb(p_i)))$, which has finitely many connected components,
each of which is an $\cF^u$-minimal component
\item[(b)] $l$ ergodic maximal measures $\mu^+_i$ with positive center exponent; each support $\supp\mu^-_i$
coincides with $\Cl(\cF^s(\Orb(q_i)))$, which has finitely many connected components,
each of which is an $\cF^s$-minimal component.
\end{itemize}
\end{main}

We also analyze how the maximal measures vary with the diffeomorphism. For the second part of the following
statement keep in mind that part (c) of Proposition~\ref{p.skeleton} characterizes when the
hyperbolic continuation $\{p_1(g),\cdots,p_k(g)\}$ of a skeleton of $f$ is a skeleton of a nearby map $g$.

\begin{main}\label{main.robust}
Let $f\in \SPH1$.
If $f$ has no hyperbolic periodic points, there exists a $C^1$-neighborhood $\cU$ of $f$ among $C^2$ diffeomorphisms
such that any $g\in \cU$ has at most two ergodic maximal measures.
Moreover, there are continuous functions $\Gamma^+, \Gamma^-:\cU \to \cP(M)$ such that, for any $g\in\cU$,
$\Gamma^+(g)$ is an ergodic maximal measure of $g$ with center exponent greater than or equal to $0$,
and
$\Gamma^-(g)$ is an ergodic maximal measure of $g$ with center exponent smaller than or equal to $0$.

Now suppose that $f$ has some hyperbolic periodic point. Let $S(f)=\{p_1,\cdots,p_k\}$ and
$S(f^{-1})=\{q_1,\cdots,q_l\}$ be skeletons of $f$ and $f^{-1}$, respectively.
Then there is a $C^1$-neighborhood $\cU$ of $f$ among $C^2$ diffeomorphisms such that the number of
ergodic maximal measures of any diffeomorphism $g\in\cU$ with negative (respectively, positive) center exponent
is smaller than or equal to $k$ (respectively, $l$).

Moreover, if the hyperbolic continuation $\{p_1(g),\cdots,p_k(g)\}$ is a skeleton of $g$ (respectively,
$\{q_1(g), \cdots, q_l(g)\}$ is a skeleton of $g^{-1}$), then $f$ and $g$ have the same number of ergodic maximal measures
with negative (respectively, positive) center exponent, and corresponding ergodic maximal measures of
the two diffeomorphisms are close to each other in the weak$^*$ topology.
\end{main}

It is worthwhile interpreting and detailing the contents of Theorems~\ref{main.skeleton} and~\ref{main.robust} in
more geometric terms, by means of the space $\MM(f)$ of all maximal measures of $f$.

Since the metric entropy is an upper semi-continuous
function of the diffeomorphism and the invariant probability measure (see \cite{LVY}), $\MM(f)$ is a non-empty compact subset of
the space of invariant probability measures, and it varies upper semi-continuously with $f$.
Since the metric entropy function is affine, $\MM(f)$ is a convex set.
Its extreme points are precisely the ergodic maximal measures, whose set we denote as $\MM_{erg}(f)$.
Thus, the ergodic components of any maximal measure are ergodic maximal measures.
Theorem~\ref{main.skeleton} gives that $\MM(f)$ is a simplex of finite dimension, for any $f\in\SPH1$.

In the first case of Theorem~\ref{main.robust}, that is, if $f$ has no hyperbolic periodic points,
$\MM(f)$ reduces to a point.
Then for every $g$ close to $f$ in the $C^1$ topology, $\MM(g)$ is either a point or a segment whose endpoints
are ergodic maximal measures with positive center exponent and negative center exponent, respectively.
Moreover, $\MM(g)$ is close to $\MM(f)$ in the weak$^*$ topology.

In the second case of Theorem~\ref{main.robust}, the set $\MM_{erg}(f)$ of extreme points of $\MM(f)$ splits into two
subsets, $\MM^-_{erg}(f)$ and $\MM^+_{erg}(f)$, consisting of the ergodic maximal measures with negative center exponent
and positive center exponent, respectively. Each $\MM^\sigma_{erg}(f)$, $\sigma=\{+, -\}$ generates a finite-dimensional
sub-simplex $\MM^\sigma(f)$. In the following, we discuss only $\MM^-(f)$: similar observations apply to $\MM^+(f)$.

We can mark each extreme point $\mu_i$ ($i=1,\cdots, k$) in $\MM^-_{erg}(f)$ by any periodic point $p_i\in\supp\mu_i$
with stable index equal to the dimension of $E^{cs}$. These periodic points form a skeleton $S=\{p_1, \dots, p_k\}$ of $f$,
and $\Cl(W^u(\Orb(p_i)))$ coincides with the support of $\mu_i$ for every $i$.
Fix $S$ and let $S(g)=\{p_1(g), \dots, p_k(g)\}$ be its hyperbolic continuation, for any nearby diffeomorphism $g$.
If $g$ is close to $f$, the space $\MM^-(g)$ is contained in a small neighborhood of $\MM^-(f)$.

There are two situations to be considered.
If there are no heteroclinic intersections between the saddle points $p_i(g)$ then $\dim(\MM^-(g))=\dim(\MM^-(f))$,
and the set $\MM^{-}_{erg}(g)$ of extreme points is close to $\MM^{-}_{erg}(f)$.
Otherwise, if heteroclinic intersections are indeed created between distinct saddles $p_i(g)$ and $p_j(g)$,
different ergodic maximal measures with negative center exponent ``merge'' with one another, so that $\dim(\MM^-(g))<\dim(\MM^-(f))$.



\section{Preliminaries}

Since each $\pi_c^{-1}(x_c)$, $x_c\in M_c$ is a circle with uniformly bounded length, and $f$ acts by homeomorphisms
on those circles, the projection $\pi_c$ preserves the topological and metric entropies:
\begin{equation}\label{eq.preservesboth}
\begin{aligned}
h_{top}(f) & = h_{top}(f_c) \text{ and} \\
h_\mu(f) & = h_{(\pi_c)_*\mu}(f_c)
\text{ for every $f$-invariant probability measure $\mu$.}
\end{aligned}
\end{equation}
In particular, $\MM(f)$  coincides with the set of $f$-invariant probability measures $\mu$ such that $(\pi_c)_*(\mu)=\nu$,
where $\nu$ denotes the maximal measure of $f_c$ (Remark~\ref{r.baseconjugate}).

The part of Theorem~\ref{main.skeleton} stating that there are finitely many ergodic maximal measures was proven
in \cite[Theorem~1]{HHTU}:

\begin{proposition}\label{p.finiteness}
Let $f\in \SPH1$, then:
\begin{itemize}
\item[(a)] either $f$  admits a unique maximal measure, with vanishing center exponent, in which case $f$ is rotation type,
\item[(b)] or $f$ has more than one ergodic maximal measure, all of them with non-vanishing central Lyapunov exponents,
not all with the same sign.
\end{itemize}
\end{proposition}

\begin{remark}\label{r.complement}
The analysis of case (a) relies on the invariance principle of \cite{AV}. One gets that the maximal
measure admits a continuous disintegration along the center foliation, and the corresponding
conditional probabilities are equivalent to Lebesgue measure on each center leaf.
Moreover, the disintegration is both $s$-invariant and $u$-invariant,
meaning that the conditional probabilities are preserved by both the stable holonomies and the unstable holonomies.
In particular, the support of the maximal measure coincides with the ambient manifold.
The diffeomorphisms satisfying the conditions in case (b) form a $C^1$ open and $C^\infty$ subset.
\end{remark}

\subsection{Unstable partitions and invariant holonomies}

Let $\{\tB_1,\cdots, \tB_k\}$ be a Markov partition for $f_c:M_c\to M_c$ such that $\nu(\cup_{1\leq i \leq k}\partial \tB_i)=0$.
For every $x_c\in \tB_i\subset M_c$, let $\cW^u_{loc}(x_c)$ be the connected component of $\cW^u(x_c) \cap \tB_i$
that contains $x_c$, and let $\txi$ be the partition of $M_c$ whose elements are those local unstable sets.
This partition $\txi$ is measurable (in the sense of Rokhlin~\cite{R}, see \cite[Chapter~5]{FET}) and increasing,
meaning that $f\txi\prec \txi$.
Let $\{\nu^u_{x_c}: x_c\}$ be a Rokhlin disintegration of $\nu$ relative to this partition.
%
%
%

For each $x_c$ and $y_c$ in the same $\tB_i$, and any $z\in \txi(x_c)$, let $\tcH^s_{x_c,y_c}(z)$
be the unique point where $\cW^s_{loc}(z)$ intersects $\txi(y_c)$.
The map $\tcH^s_{x_c,y_c}: \txi(x_c)\to\txi(y_c)$ thus defined is called \emph{stable holonomy map}
from $x_c$ to $y_c$.

\begin{lemma}\label{l.sjac}
$(\tcH^s_{x_c,y_c})_* \nu^u_{x_c}=\nu^u_{y_c}$ for $\nu$-almost any points $x_c$ and $y_c$ in the same Markov set $\tB_i$.
\end{lemma}

This is well known (see \cite{B}, for example) and can be proven as follows.
Up to conjugacy, we may view $f_c$ as a linear Anosov torus diffeomorphism, and then the stable holonomy maps
$\tcH^s_{x_c,y_c}$ are affine.
Moreover, $\nu$ corresponds to the Lebesgue area, and each $\nu^u_{x_c}$ corresponds to the normalized Lebesgue length
along the plaque $\txi(x_c)$.
Thus the lemma follows from the fact that affine maps preserve normalized Lebesgue length.

In general, Rokhlin disintegrations are defined up to zero measure sets only.
However, Lemma~\ref{l.sjac} ensures that in the present setting there is a canonical choice for which the conditional
measure is defined on every plaque $\txi(z_c)$ and depends continuously on $z_c$ on every $\tB_i$.
Indeed, for each $i$, choose $x^i_c \in \tB_i$ as in the lemma, and then define
$$
\hnu^u_{z_c} = (\tcH^s_{x^i_c,z_c})_* \nu^u_{x^i_c} \text{ for every } z_c \in \tB_i
$$
(the definitions may not coincide on overlapping boundaries of different Markov sets, but that need not
concern us too much, since the assumptions ensure that such overlaps have zero measure for $\nu$).
By the lemma, $\hnu^u_{z_c}=\nu^u_{z_c}$ for $\nu$-almost $z_c \in \tB_i$, and so $\{\hnu^u_{z_c}: z_c\}$
is also a disintegration of $\nu$.
Just replace the initial disintegration with this one, and notice that it does not depend on the choices of the points $x^i_c$.

Then we also have that (keep in mind that $\txi$ is an increasing partition)
\begin{equation}\label{eq.baseMarkov}
(f_c)_*\left(\nu^u_{f_c^{-1}(x_c)}\mid_{f_c^{-1}(\txi(x_c))}\right)=\nu^u_{f^{-1}(x_c)}\left(f_c^{-1}(\txi(x_c))\right)\nu^u_{x_c}
\end{equation}
for every point $x_c\in M_c$. It is easy to see that,
for a linear Anosov map, the factor $\nu^u_{f^{-1}(x_c)}(f_c^{-1}(\txi(x_c)))$
takes only finitely many values. This fact will be useful later.

\begin{proposition}\label{p.limit}
For any $z_c\in M_c$ and $\Delta\subset \cW^u_{loc}(z_c)$ with positive $\nu^u_{z_c}$-measure,
$$
\lim_n \frac{1}{n}\frac{1}{\nu^u_{z_c}(\Delta)}\sum_{i=1}^{n}(f_c^i)_* (\nu^u_{z_c}|\Delta)=\nu.
$$
\end{proposition}

\begin{proof}
Up to conjugacy, we may view $f_c$ as a linear Anosov torus diffeomorphism.
Then $\nu$ is the Lebesgue area and $\nu^u_{z_c}$ the normalized Lebesgue length on the plaque $\txi(z_c)$.
Then the statement reduces to \cite[Section~11.12]{BDV}.
\end{proof}

Next, denote $B_i=(\pi_c)^{-1}(\tB_i)$ for each $i$. For every $x\in B_i \subset M$, let $\cF^u_{loc}(x)$ be
the connected component of $\cF^u(x)\cap B_i$ that contains $x$.
Denote by $\xi$ the partition of $M$ whose atoms are those local unstable leafs.
It is clear that $\xi$ is an increasing partition for $f$. For every $x\in M$, denote $x_c=\pi_c(x)\in M_c$.
Then $\pi_c\mid_{\xi(x)}$ is a homeomorphism from $\xi(x)$ onto $\txi(x_c)$.
By a slight abuse of language, we denote by $\nu^u_x$ the probability measure on $\xi(x)$ such that
$(\pi_c)_*(\nu^u_x)=\nu^u_{x_c}$.

For $x$ and $y$ in the same $B_i$, and any $z\in\xi(x)$ denote by $\cH^{cs}_{x,y}(z)$ the unique point in the
intersection of $\cF^{cs}_{loc}(z)$ with $\xi(y)$.
The map $\cH^{cs}_{x,y}:\xi(x)\to\xi(y)$ defined in this way is called the \emph{center-stable holonomy map}
from $x$ to $y$. Lemma~\ref{l.sjac} and the remarks about the canonical choice of a
disintegration following it immediately yield:

\begin{lemma}\label{l.csjac_Markov_finiteness}~
\begin{itemize}
\item[(a)] $(\cH^{cs}_{x,y})_* \nu^u_{x}=\nu^u_{y}$ for any two points $x,y\in B_i$.
\item[(b)] for every $x\in M$,
$$
f_* \left(\nu^u_{f^{-1}(x)}\mid_{f^{-1}(\xi(x))}\right)=\nu^u_{f^{-1}(x)}\left(f^{-1}(\xi(x))\right)\nu^u_{f(x)}
$$

\item[(c)] the factor $\nu^u_{f^{-1}(x)}(f^{-1}(\xi(x)))$, $x\in M$ takes only finitely many values.
\end{itemize}
\end{lemma}

\subsection{Partial entropy, $\nu$-Gibbs $u$-states and $u$-invariant measures}

Let $\mu$ be any probability measure invariant under $f\in \SPH1$.
We call $\mu$ a \emph{$\nu$-Gibbs $u$-state} if it admits a disintegration $\{\mu^u_x: x\}$
with respect to the partition $\xi$ such that $\mu^u_x=\nu^u_x$ for $\mu$-almost every point $x\in M$.
This is a variation of a general notion due to \cite{PS}:
the main difference with respect to the standard definition (see \cite[Chapter~11]{BDV},
for instance) is that here we replace Lebesgue with $\nu$ as the reference measure.
Let $\Gibb^u_\nu(f)$ denote the space of $\nu$-Gibbs $u$-states of $f$.

\begin{lemma}\label{l.MM}
$\Gibb^u_\nu(f) \subset \MM(f)$ for every $f\in\SPH1$.
\end{lemma}

\begin{proof}
Let $\mu\in\Gibb^u_\nu(f)$. Then, by definition, the disintegration of $(\pi_c)_*(\mu)$ along the
partition $\txi$ coincides with $\{\nu^u_{x^c}: x^c\}$.
Thus $(\pi_c)_*(\mu)=\nu$, and so $\mu\in\MM(f)$.
\end{proof}

Let $\mu$ be any probability measure invariant under $f\in \SPH1$.
Denote by $\{\mu^c_x:x\}$ the disintegration of $\mu$ along the center foliation of $f$.
We say that $\mu$ is \emph{$u$-invariant} if there is a full $\nu$-measure set $\Delta \subset M_c$
such that for any $x,y\in M$ with $x_c,y_c\in \Delta$ belonging to the same unstable leaf,
the unstable holonomy map takes $\mu^c_x$ to $\mu^c_y$.
According to \cite[Proposition 5.4]{TY}, this turns out to be the same as the previous notion:

\begin{proposition}\label{p.iff}
Let $f\in\SPH1$.
Then an $f$-invariant probability measure $\mu$ is $u$-invariant if and only if it is a $\nu$-Gibbs $u$-state.
\end{proposition}

The following fact was proven in \cite[Corollary~4.3]{AV}, see also  \cite[Corollary 2.3]{TY}:

\begin{proposition}\label{p.AVuinvariant}
Let $f\in\SPH1$.
If $\mu$ is an ergodic $f$-invariant probability measure of $f\in \SPH1$ with non-positive center Lyapunov exponent
then it is $u$-invariant.
\end{proposition}

There is a notion dual to $u$-invariance, called \emph{$s$-invariance}, where one requires the
disintegration of $\mu$ along the center foliation to be invariant under stable holonomies instead,
at almost every point. The following fact was proven in \cite[Section~5]{HHTU}, based on the
invariance principle of \cite{AV}; similar ideas appear in \cite[Section~5.2]{VY}.

\begin{proposition}\label{p.rigidityofsuinvariant}
Let $f\in \SPH1$. If there exists some $\mu\in \MM(f)$ which is $u$-invariant, $s$-invariant and ergodic
then $f$ is rotation type and $\MM(f)=\{\mu\}$.
\end{proposition}

As before, let $\mu$ be an invariant probability measure for $f\in \SPH1$.
The \emph{ partial entropy of $\mu$ along the unstable foliation $\cF^u$} is defined by
$$
h_\mu(f,\cF^u)=H_\mu(f^{-1}\xi\mid \xi).
$$
Theorem A in \cite{TY} provides the following very useful criterion for an invariant measure
to be a $\nu$-Gibbs $u$-state:

\begin{proposition}\label{p.uinvariant1}
$h_\mu(f,\cF^u)\leq h_\nu(f_c,\cW^u)$ for any $f$-invariant probability measure $\mu\in\MM(f)$.
Moreover, the equality holds if and only if $\mu$ is a $\nu$-Gibbs $u$-state.
\end{proposition}

The following upper semi-continuity property of partial entropy, proven in \cite[Theorem~D]{Y},
is important to understand the dependence of $\nu$-Gibbs $u$-states with respect
to the diffeomorphism:

\begin{proposition}\label{p.continuityofpartialentropy}
Let $f_n$ be a sequence of diffeomorphisms in $\SPH1$ converging to some
$f$ in the $C^1$ topology, and $\mu_n$ be $f_n$-invariant probability measures converging,
in the weak$^*$ topology, to some $f$-invariant probability measure $\mu$. Then
$$
\limsup_n h_{\mu_n}(f_n, \cF^u_n) \leq h_{\mu}(f,\cF^u)
$$
where $\cF_n^u$ is the unstable foliation of $f_n$.
\end{proposition}

\subsection{Space $\nu$-Gibbs $u$-states}

The next two propositions contain useful properties of the space $\Gibb^u_\nu(f)$
of $\nu$-Gibbs $u$-states of $f$.

\begin{proposition}\label{p.Gibbsustates}~
\begin{itemize}
\item[(a)] $\Gibb^u_\nu(f)$ is a non-empty compact subset of $\MM(f)$;
\item[(b)] if $\mu \in \Gibb^u_{\nu}(f)$ then almost every ergodic component of $\mu$ is also a $\nu$-Gibbs $u$-state;
\item[(c)] $\Gibb^u_{\nu}(f)$ varies upper semi-continuously with respect to the diffeomorphism $f$ in the $C^1$ topology.
\item[(d)] the support of every $\mu\in\Gibb^u_{\nu}(f)$ is $u$-saturated.
\end{itemize}
\end{proposition}

\begin{proof}
First, we prove claim (a).
By Proposition~\ref{p.finiteness}, there exists some ergodic maximal measure $\mu$ with non-positive
center exponent. By Proposition~\ref{p.AVuinvariant}, such a $\mu$ is a $u$-invariant probability measure.
So, by Proposition~\ref{p.iff}, it is a $\nu$-Gibbs $u$-state.
This proves that $\Gibb^u_\nu(f)$ is non-empty.
The fact that $\Gibb^u_{\nu}(f)$ is compact was proven in \cite[Proposition 6.2]{TY}.
We have already seen in Lemma~\ref{l.MM} that $\Gibb^u_\nu(f)\subset \MM(f)$.

Claim (b) is contained in \cite[Proposition 6.2]{TY}.

Next we prove claim (c).
Let $f_n\in \SPH1$ be a sequence converging to some $f\in \SPH1$ in the
$C^1$ topology, and $\mu_n\in \Gibb^u_\nu(f_n)$ converge to some ($f$-invariant) probability measure $\mu$ in the weak$^*$ topology.
We need to show that $\mu\in \Gibb^u_\nu(f)$.

Let $\cF_n^c$ denote the center foliation of $f_n$, and $f_{n,c}$ be the map induced by $f$ on the quotient space $M/\cF_n^c$.
Observe that $f$ and $f_n$ are leaf conjugate, for every large $n$, since each is leaf conjugate to its algebraic
part (see \cite[Theorem~1.1]{H} and \cite[Theorem~1.6]{HP}), and the algebraic parts coincide.
Thus, $f_c$ and $f_{n,c}$ are topologically conjugate (this also follows from Remark~\ref{r.baseconjugate}),
and so $h_{top}(f_c)=h_{top}(f_{n,c})$ for every large $n$.
Using \eqref{eq.preservesboth} and Lemma~\ref{l.MM}, it follows that
$$
h_{\mu_n}(f_n)=h_{top}(f_n) = h_{top}(f_{n,c}) = h_{top}(f_c) = h_{top}(f)
$$
for every large $n$. Since partially hyperbolic diffeomorphisms with 1-dimensional center are away from homoclinic
tangencies, \cite[Corollary~C]{LVY} gives that their metric entropies vary upper semi-continuously in the $C^1$ topology.
Thus,
$$
h_\mu(f) \ge \limsup_{n} h_{\mu_n}(f_n)=h_{top}(f),
$$
which means that $\mu\in \MM(f)$.

Let $\nu_n$ denote the (unique) maximal measure for $f_{n,c}$ (recall Remark~\ref{r.complement}).
As observed previously, $f_c$ and $f_{n,c}$ are conjugate by some homeomorphism $g_n$. Thus
$$
h_{(g_n)_*(\nu)}(f_{n,c}) = h_{\nu}(f_c) = h_{top}(f_c) = h_{top}(f_{n,c}),
$$
and so $(g_n)_*(\nu)=\nu_n$. Moreover, observing that $g_n$ maps $\cW^u$ to $\cW^u_n$,
$$
h_{\nu_n}(f_{n,c},\cW^u_n) = h_{\nu}(f_c,\cW^u).
$$
Since $\mu_n\in \Gibb^u_\nu(f_n)$, Proposition~\ref{p.uinvariant1} gives that
$$
h_{\mu_n}(f_n,\cF^u_n) = h_{\nu_n}(f_{n,c},\cW^u_n) = h_{\nu}(f_c,\cW^u).
$$
On the other hand, Proposition~\ref{p.continuityofpartialentropy} gives that
$$
\limsup_n h_{\mu_n}(f_n,\cF^u_n) \leq h_\mu(f,\cF^u).
$$
These two inequalities imply that $h_\mu(f,\cF^u)\geq h_\nu(f_c,\cW^u)$.
Using Proposition~\ref{p.uinvariant1}, we conclude that $h_\mu(f,\cF^u) = h_\nu(f_c,\cW^u)$ and
$\mu\in \Gibb^u_\nu(f)$.

To prove claim (d), it suffices to show that $\xi(x)\subset \supp\mu$ for $\mu$-almost every $x$.
By the definition of $\nu$-Gibbs $u$-state, $(\pi_c)_*(\mu^u_x)=\nu^u_{x_c}$ for every $x$.
Note that $\nu^u_{x_c}$ is supported on the whole $\txi(x_c)$, as it corresponds (in the sense of
Remark~\ref{r.baseconjugate}) to the normalized Lebesgue measure on $\txi(x_c)$.
Thus $\supp \mu^u_x=\xi(x)$. Moreover, for $\mu$-almost every point $x$, $\mu^u_x$-almost every point
is a regular point of $\mu$, hence $\xi(x)=\supp(\mu^u_x)\subset \supp\mu$.
\end{proof}

\begin{proposition}\label{p.nupositivesetinleaf}
For every $x\in M$, any weak$^*$ limit of the sequence of probability measures
$$
\frac{1}{n}\sum_{j=0}^{n-1} (f^j)_* (\nu^u_x)
$$
is a $\nu$-Gibbs $u$-state.
\end{proposition}

\begin{proof}
Let $\Phi_i: I \times D \to B_i$ be a foliation chart for the foliation $\cF^u \mid_{B_i}$, that is,
a homeomorphism such that each $\Phi(\cdot,\theta)$ maps the interval $I$ diffeomorphically to an element
of the partition $\xi$ inside $B_i$.
This may be chosen in such a way that the image of each $\Phi_i(a,\cdot)$ is contained in a leaf of $\cF^{cs}$, and we do so.
We shall use on each $B_i$ the coordinates defined by the corresponding chart $\Phi$.
Observe that in these coordinates the partition $\xi \mid_{B_i}$ consists of the horizontal line segments
$I\times\{\theta\}$. Moreover, the assertion in Lemma~\ref{l.csjac_Markov_finiteness}(a)
$$
(\cH^{cs}_{x,y})_*(\nu^u_x)=\nu^u_y \text{ for any $x,y\in B_i$,}
$$
means that the disintegration of $\nu^u$ along $\xi\mid_{B_i}$ is constant:
the conditional measure is the same on every horizontal segment. Let $\hat\nu_i$ denote this measure,
which we may also view as a probability measure on $I$.
Then, by Lemma~\ref{l.csjac_Markov_finiteness}(b), every
$$
\frac{1}{n}\sum_{j=0}^{n-1}(f^j)_*(\nu^u_x)
$$
is a finite linear combination of such measures $\hat\nu_i$. It follows that every accumulation point for this
sequence is a sum of measures of the form $\hat\nu_i\times\xi_i$ on each $B_i \approx I \times D$.
Thus, its conditional measures are precisely the $\nu_i^u$, as we wanted to prove.
\end{proof}

\section{Maps with $\nu$-mostly contracting center}

We say that $f\in \SPH1$ has \emph{$\nu$-mostly contracting center} if for every $x\in M$ there is a positive $\nu^u_x$-measure
set $\Delta\subset \xi(x)$ such that
\begin{equation}\label{eq.mostlycontractiny}
\limsup_n \frac{1}{n}\log \|Df^n \mid_{E^c(y)}\|<0
\end{equation}
for every $y\in \Delta$. This is a straightforward adaptation of the notion of  diffeomorphisms with mostly contracting center
introduced in \cite{BV}: we just replace Lebesgue measure along unstable leaves with the conditional measures of $\nu^u$.
The next couple of propositions also have analogues in the classical setting (see \cite{BV,DVY}).
The proofs in our present context will be given in the following two subsections.

\begin{proposition}\label{p.criterionofmostlycontracting}
A diffeomorphism $f\in\SPH1$ has $\nu$-mostly contracting center if and only if every ergodic $\nu$-Gibbs $u$-state
has negative center exponent.
\end{proposition}

\begin{proposition}\label{p.isolated}
If $f\in\SPH1$ has $\nu$-mostly contracting center then it has finitely many ergodic $\nu$-Gibbs $u$-states,
and their supports are pairwise disjoint.
Moreover, each support is the union of finitely many minimal components of the unstable foliation.
\end{proposition}

\subsection{Proof of Proposition~\ref{p.criterionofmostlycontracting} }

Suppose that $f$ has $\nu$-mostly contracting center and let $\mu$ be any ergodic $\nu$-Gibbs $u$-state of $f$.
By ergodicity,
$$
\lim_n \frac{1}{n}\log \|Df^n \mid_{E^c(y)}\|
$$
coincides with the center exponent $\lambda^c(\mu)$ at $\mu$-almost every point or, equivalently, at $\mu^u_x$-almost every point
for $\mu$-almost every $x$. Thus, since it is assumed that $\mu_x^u=\nu_x^u$ at $\mu$-almost every $x$, the assumption
\eqref{eq.mostlycontractiny} implies that $\lambda^c(\mu)<0$ is negative.

Conversely, suppose that all ergodic $\nu$-Gibbs $u$-states of $f$ have negative center exponent.
Using  Proposition \ref{p.Gibbsustates}(a), for any $x\in B_i$ there exists an increasing sequence $(n_k)_k$
of positive integers such that
$$
\lim_{k\to \infty} \frac{1}{n_k}\sum_{j=0}^{n_k-1} f^j_* \nu^u_x
$$
converges to some $\nu$-Gibbs $u$-state $\mu$. By part (b) of that same proposition, almost every ergodic component of $\mu$
is a $\nu$-Gibbs $u$-state. It is clear that $\supp\mu$ is a full measure set for almost every ergodic component of $\mu$.
Thus there exists an ergodic component $\tilde\mu$ of $\mu$ that is a $\nu$-Gibbs $u$-state and satisfies $\tilde\mu(\supp\mu)=1$.

We call \emph{basin} of a measure the set of points whose (forward) time averages converge to that measure.
By ergodicity, the basin $\Basin(\tilde\mu)$ has full $\tilde\mu$-measure.
By definition, $\tilde\mu^u_x=\nu^u_x$ for a full $\tilde\mu$-measure set $\Gamma\subset\Basin(\tilde\mu)\cap\supp\mu$.
By assumption, the center exponent $\lambda^c(\tilde{\mu})$ is negative, and so $\tilde\mu$-almost every point has
a Pesin local stable manifold with dimension equal to $\dim E^{cs}$. It is no restriction to suppose that this
holds for every $x\in\Gamma$ (reducing this set if necessary).
Let $\tilde\Gamma\subset\Gamma$ be a full $\tilde\mu$-measure subset such that $\nu_x(\Gamma)=1$ for every $x\in\tilde\Gamma$.
Next, fix $z \in \tilde\Gamma$ such that $\nu_{\tilde x}(\Gamma)=1$ and, consequently, there exists some
positive $\nu_z$-measure set $K\subset\xi(z)$ consisting of points with Pesin local stable manifolds of size uniformly bounded from below.
Observe that the Pesin local stable manifolds are contained in the corresponding center-stable leaves,
and that the holonomy induced by the center-stable foliation preserves the family of reference measures $\nu^u_x$
(see part (a) of Lemma~\ref{l.csjac_Markov_finiteness}).
Thus, there exists $\delta>0$ such that,  for every $y$ in the $\delta$-ball around $z$, the local stable manifolds through
the points of $K$ intersect $\xi(y)$ on a subset whose $\nu_y^u$-measure is independent of $y$.
Notice that all these Pesin local stable manifolds are contained in $\Basin(\tilde\mu)$.
In particular, $\nu^u_y(\Basin(\tilde\mu))$ is positive for any $y$ in the $\delta$-ball around $z$.
This $\delta$-ball has positive $\mu$-measure, since $z\in\supp\mu$. So, by weak$^*$ convergence,
$$
\lim_k \frac{1}{n_k }\sum_{j=0}^{n_k-1} (f^j)_*\nu^u_x(B_\delta(z)) \ge \mu\left(B_\delta(z)\right) > 0.
$$
This ensures that there are $j\ge 0$ and $y \in B_\delta(z)\cap f^j(\xi(x))$ such that
$\Basin(\tilde{\mu})\cap\xi(y)$ has positive $\nu^u_y$-measure.
Since $\Basin(\tilde{\mu})$ is invariant under iteration, the Markov property in Lemma~\ref{l.csjac_Markov_finiteness}(b),
implies that
$$
\nu^u_x(\Basin(\tilde{\mu}))
= \nu^u_{f^{-j}(y)}(\Basin(\tilde{\mu}))>0
$$
Finally, the Birkhoff ergodic theorem asserts that
$$
\begin{aligned}
\lim_n \frac{1}{n} \log\|Df^n\mid_{E^c(w)}\|
& =\lim_n \frac{1}{n}\sum_{j=0}^{n-1}\log\|Df\mid_{E^c(f^j(w))}\|\\
&=\int \log \|Df\mid_{E^c}\|d\tilde{\mu}
= \lambda^c(\tilde{\mu})<0.
\end{aligned}
$$
for every $w \in \xi(x) \cap \Basin(\tilde{\mu})$. Hence $f$ is mostly contracting.

This completes the proof of Proposition~\ref{p.criterionofmostlycontracting}.

\subsection{Proof of Proposition  \ref{p.isolated}\label{ss.mostlycontracting}}

The proof consists of several lemmas.

\begin{lemma}\label{l.finiteness}
There are only finitely many ergodic $\nu$-Gibbs $u$-states of $f$.
\end{lemma}

\begin{proof}
Suppose there are infinitely many ergodic $\nu$-Gibbs $u$-states $\mu_n$. By Proposition \ref{p.Gibbsustates}(a)
we may assume that $(\mu_n)_n$ converges to some $\nu$-Gibbs $u$-state $\mu_0$.
By Proposition \ref{p.Gibbsustates}(b), almost every ergodic component $\tilde{\mu}$ of $\mu_0$ is a $\nu$-Gibbs
$u$-state. Moreover, $\tilde{\mu}(\supp\mu)=1$ and so $\supp(\tilde{\mu}) \subset \supp(\mu)$.
By Proposition \ref{p.criterionofmostlycontracting}, the center exponent of $\tilde{\mu}$ is negative.
Thus, $\tilde\mu$-almost every point has a Pesin local stable manifold of dimension $\dim E^{cs}$.
So, in view of the definition of $\nu$-Gibbs $u$-state, for $\tilde\mu$-almost every $z$ there exists a set
$K\subset\xi(z)\cap\Basin(\tilde\mu)$ with positive $\nu_z^u$-measure and such that its points have Pesin local
stable manifolds with size bounded from zero.
Keep in mind that these Pesin local stable manifolds are contained in $\Basin(\tilde\mu)$ and the holonomy
induced by the center stable foliation preserves the family of conditional measures $\nu_y^u$.
Let such a point $z$ be fixed from now on, and $\delta>0$ be small enough that, for every $y\in B_\delta(z)$,
$\xi(y)$ intersects $\Basin(\tilde\mu)$ on a subset with $\nu^u_y$-measure equal to $\nu^u_z(K)>0$.
It is clear that $\mu_n(B_\delta(z))$ is positive for every large $n$. Let $n$ be fixed, large enough.
By ergodicity, $\Basin(\mu_n)$ has full $\mu_n$-measure or, equivalently, full $\mu_{n,y}$-measure for
$\mu_n$-almost every $y$. On the other hand, for $\mu_n$-almost every $y\in B_\delta(z)$ we have that
$\mu^u_{n,y}=\nu^u_{y}$ which, in view of the previous observations, implies that $\mu_{n,y}(\Basin(\tilde\mu))>0$.
Combining these two observations, we see that the basins of $\mu_n$ and $\tilde{\mu}$ intersect.
Thus $\mu_n=\tilde{\mu}$, a contradiction.
\end{proof}

Now we are going to show that
\begin{lemma}\label{l.disjoint}
The supports of different ergodic $\nu$-Gibbs $u$-states are disjoint.
\end{lemma}
\begin{proof}
Suppose that there are ergodic $\nu$-Gibbs $u$-states $\mu_1$ and $\mu_2$ whose supports are not disjoint,
and let $x\in \supp \mu_1 \cap \supp \mu_2$. Then, as shown in the proof of
Proposition~\ref{p.criterionofmostlycontracting}, there is $\Gamma\subset \xi(x)$ with positive
$\nu^u_x$-measure such that the Pesin local stable manifolds of its points dimension  is equal to $\dim \cF^{cs}$
and have uniform size. Choose points $x_1$ and $x_2$ close to $x$ such that their conditional measures
along unstable leaves satisfy:
$$
\mu^u_{1,x_1} = \nu^u_{x_1} \text{ and } \mu^u_{2,x_2} = \nu^u_{x_2}.
$$
Assume furthermore that $x_1$ and $x_2$ are such that $\mu^u_{i,x_i}$-almost every point belongs to
the basin of $\mu_i$, for $i=1, 2$. Recall that the basins of both measures are saturated by Pesin
local stable manifolds.
By Lemma~\ref{l.csjac_Markov_finiteness}(a), the local stable holonomy preserves the family of conditional
measures $\nu^u_x$. It follows that the basin of $\mu_1$ and $\mu_2$ intersect, and hence $\mu_1=\mu_2$.
This contradiction proves the claim..
\end{proof}

It remains to show that the support of every ergodic $\nu$-Gibbs $u$-state $\mu$ consists of finitely
many connected components, and each of them is an $\cF^u$-minimal component.

\begin{lemma}\label{l.existenceperiodicpoint}
There is a hyperbolic periodic orbit $\Orb(p)$ with stable index $\dim E^{cs}$ contained in $\supp\mu$.
\end{lemma}

\begin{proof}
This is a consequence of Katok's closing lemma. Indeed, by \cite{K} there one can find $x\in\supp\mu$
and a periodic point $p$ such that $W^s(p)$ has dimension equal to $\dim E^{cs}$ and intersects the
unstable leaf of $x$ transversely at some point $y$. Since $\supp \mu$ is $u$-saturated (by part (d) of
Proposition~\ref{p.Gibbsustates}), the point $y\in W^s_{loc}(p) \cap \supp(\mu)$.
Let $\kappa$ be the minimum period of $p$.
Since $\supp \mu$ is invariant and closed, it follows that $p=\lim_{n} f^{n\kappa}(y)$ also belongs to
$\supp(\mu)$.
\end{proof}

\begin{lemma}\label{l.closureofunstable}
Assuming $\mu$ is ergodic, $\supp \mu=\Cl(\cF^u(\Orb(p)))$.
\end{lemma}
\begin{proof}
It is clear that $\supp \mu \supset \Cl(\cF^u(\Orb(p)))$, since $\supp\mu$ is closed and $\cF^u$-saturated.
To prove the converse, let $\hat{\mu}$ be any accumulation point of the sequence
$$
\frac{1}{n}\sum_{i=0}^{n-1}f^i_* \nu^u_p.
$$
It is clear that $\supp \hat{\mu}$ is contained in $\Cl(\cF^u(\Orb(p)))$ and, thus, is contained in $\supp\mu$.
By parts (d) and (b) of Proposition~\ref{p.Gibbsustates}, $\hat{\mu}$ is a $\nu$-Gibbs $u$-state, and so are
almost all its ergodic components. For any ergodic component $\tilde{\mu}$,
it is clear that $\supp\tilde{\mu} \subset \supp\hat{\nu} \subset  \Cl(\cF^u(\Orb(p))) \subset\supp\mu$.
Recalling that $\mu$ and $\tilde{\mu}$ are ergodic, it follows from Lemma~\ref{l.disjoint} that $\tilde\mu = \mu$.
Thus $\hat\mu=\mu$, and the claim follows immediately.
\end{proof}

Let $p$ be a hyperbolic periodic point as in Lemma~\ref{l.existenceperiodicpoint}, and let $\kappa$ be its period.
Note that $F=f^{\kappa}$ also has $\nu$-mostly contracting center: the assumption \eqref{eq.mostlycontractiny}
remains valid if we replace $f$ by any positive iterate, clearly. Since $\mu$ is $f$-ergodic, its ergodic composition
for $F$ has the form
$$
\mu=\frac{1}{l} \left(\mu_0+\cdots+f^{l-1}_*(\mu_0)\right)
$$
for some $\mu_0$ such that $f^l_*(\mu_0)=\mu_0$ and some divisor $l$ of $\kappa$.
By Proposition~\ref{p.Gibbsustates}(b), the $f_*^i\mu_0$ are $\nu$-Gibbs $u$-states.
Then, by Lemma~\ref{l.disjoint}, their supports are pairwise disjoint.

It remains to show that each $\supp(f^i_*(\mu_0))$ is an $\cF^u$-minimal component. It is no restriction to suppose that
$i=0$ and $p\in \supp \mu_0$, and we will do so.
For every point $x\in \supp \mu_0$, the same argument as in the proof of Lemma~\ref{l.closureofunstable}, shows that
$$
\frac{1}{n}\sum_{i=0}^{n-1}F^i_* \nu^u_x \text{ converges to } \mu_0.
$$
It follows that $\cup_{n \ge 0} F^{n}(\cF^u(x))$ is dense in $\supp \mu_0$ and, in particular, $\cF^u(p)$ is dense in $\supp \mu_0$.
Keeping in mind that $\cF^u$ is a continuous foliation, it follows that there is $n_x\ge 1$ such that $F^{n_x}(\cF^u(x))$ intersects
$W^s(p)$ transversely at some point. Taking backward iterates, we get that $\cF^u(x)$ also intersects $W^s(p)$ transversely.
By continuity of the unstable foliation, it follows that there exist $R_x>0$ and a neighborhood $U_x$ of $x$ such $\cF^u(y)$
intersects $W^s_{R_x}(p)$ (the $R_x$-neighborhood of $p$ inside the stable manifold) transversely at some point, for every $y\in U_x$.
Let $\{U_{x_1},\cdots, U_{x_j}\}$ be a finite cover of $\supp \mu_0$ by such neighborhoods and $R=\max\{R_{x_1},\dots, R_{x_j}\}$.
Then $\cF^u(y)$ intersects $W^s_{R}(p)$ transversely at some point, for every $y\in \supp \mu_0$. This also implies that for
each $n\ge 1$ there exists $a_n\in\cF^u(F^{-n}(y)) \cap W^s_{R}(p)$ a point of transverse intersection.
Since $F^n(a_n)\to p$, we conclude that $p\in \Cl(\cF^u(y))$, which implies that $\cF^u(p)\subset \Cl(\cF^u(y))$.
Since we have already shown that $\cF^u(p)$ is dense in $\supp \mu_0$, this finishes the proof.

\section{Proof of Theorem~\ref{main.skeleton}}

This is similar to the proof of \cite[Theorem~A]{DVY}, which deals with physical measures of diffeomorphisms with mostly contracting center.
The case when $f$ is of rotation type is covered by Proposition~\ref{p.finiteness}.
So, we only need to consider the case when $f$ has some hyperbolic periodic point, and all the ergodic maximal measures of $f$
have non-vanishing center exponent. Let us focus on the ergodic maximal measures with negative center exponent, corresponding
to part (a) of the theorem. Part (b) is entirely analogous.
Recall that $\MM^-(f)$ denotes the simplex generated by the finitely many ergodic maximal measures with negative center exponent.

\begin{lemma}\label{l.maximalmeasureandgibbsustates}
$\MM^-(f)=\Gibb^u_\nu(f)$.
\end{lemma}

\begin{proof}
Combining Propositions~\ref{p.iff} and~\ref{p.AVuinvariant}, we get that every ergodic maximal measure with negative center
exponent is a $\nu$-Gibbs $u$-state, that is, $\MM^-(f)\subset \Gibb^u_\nu(f)$.

Suppose that there is some ergodic $\nu$-Gibbs $u$-state $\mu$ with non-negative center exponent.
On the one hand, by Proposition~\ref{p.iff}, $\mu$ is $u$-invariant since it is a $\nu$-Gibbs $u$-state.
On the other hand, by Proposition~\ref{p.AVuinvariant}, the assumption on the center exponent ensures that $\mu$ is $s$-invariant.
Hence, $\mu$ is both $s$- and $u$-invariant.
By Proposition~\ref{p.rigidityofsuinvariant}, that implies that $f$ is of rotation type, which contradicts the assumptions.
This contradiction proves that every ergodic $\nu$-Gibbs $u$-state has negative center exponent.
Using Proposition~\ref{p.Gibbsustates}(a), it follows that $\Gibb^u_\nu(f)\subset\MM^-(f)$.
\end{proof}

Combining Lemma~\ref{l.maximalmeasureandgibbsustates} with Proposition~\ref{p.criterionofmostlycontracting} we immediately get
\begin{corollary}\label{c.mostlycontracting}
$f$ has $\nu$-mostly contracting center.
\end{corollary}

Thus, using Proposition~\ref{p.isolated} we get that $f$ has finitely many ergodic $\nu$-Gibbs $u$-states.
Their supports are $u$-saturated, pairwise disjoint, and each one is a union of finitely many minimal components of the unstable foliation.
Denote by $\mu^-_1, \dots, \mu^-_k$ the ergodic maximal measures with negative center exponent.
For every $1\leq i \leq k$, let $p_i\in\supp\mu_i^-$ be a hyperbolic periodic point whose stable manifold has dimension equal to $\dim E^{cs}$
(see Lemmas \ref{l.existenceperiodicpoint} and \ref{l.closureofunstable}).

\begin{lemma}\label{l.showskeleton}
$\{p_1,\dots, p_k\}$ is a skeleton for $f$.
\end{lemma}

\begin{proof}
By Proposition~\ref{p.isolated}, the support of each $\mu_i$ coincides with the closure of $\Orb(p_i)$.
Since the supports are pairwise disjoint, it follows from the inclination lemma (Palis' $\lambda$-lemma)
that $W^s(\Orb(p_i))\cap W^u(\Orb(p_j))=\emptyset$ for any $1\leq i \neq j \leq k$.
To complete the proof it remains to show that the unstable leaf through any point $x\in M$ intersects
$W^s(\Orb(p_i))$ for some $i$. Let $\mu$ be any accumulation point of the sequence
$$
\frac{1}{n}\sum_{i=1}^n (f^i)_*(\nu^u_{x}).
$$
By Proposition~\ref{p.nupositivesetinleaf}, $\mu$ is a $\nu$-Gibbs $u$-state.
By Proposition~\ref{p.Gibbsustates}(b), $\mu$ is a convex combination $a_1\mu_1^- +  \cdots + a_k \mu_k^-$.
Fix any $j$ such that $a_j>0$. Then there exist $n$ arbitrarily large and $r>0$ such that
$$
f^n_*(\nu^u_x)(B_r(p_j))>0.
$$
This implies that $f^n(\cF^u_{loc}(x))$ intersects $B_r(p_j)$, which on its turn implies that
$f^n(\cF^u_{loc}(x))$ has some intersection with $W^s(\Orb(p_j))$. The intersection is transverse,
since the dimensions are complementary.
By taking backward iterates, we get that $\cF^u_{loc}(x)$ has some transverse intersection with $W^s(\Orb(p_j))$.
\end{proof}

The proof of Theorem~\ref{main.skeleton} is complete.

\section{Proof of Theorem~\ref{main.robust}}

\subsection{Rotation type case}

If $f\in\SPH1$ has no hyperbolic periodic points then, by Theorem~\ref{main.skeleton}(a), it has a unique ergodic maximal measure $\mu$,
and it has vanishing center exponent.
Moreover, by Proposition~\ref{p.finiteness}(a), the support $\supp\mu$ is the whole ambient manifold.
By Proposition~\ref{p.Gibbsustates}(a), there is a unique $\nu$-Gibbs $u$-state and it coincides with $\mu$.
Thus, by Proposition~\ref{p.nupositivesetinleaf}, for any $x\in M$,
\begin{equation}\label{eq.denseusegment}
\lim_{n\to \infty}\frac{1}{n}\sum f^i_*(\nu^u_x)=\mu
\end{equation}
and, in particular, $\cup_{n\geq 0} f^n (\cF_1^u(x))$ is dense in the ambient manifold $M$.
Recall, from \cite{D} that accessibility is a $C^1$ open property.

\begin{proposition}\label{p.uniqueminimalcomponent}
Every $g$ in a $C^1$-neighborhood of $f$ has a unique $u$-saturated compact invariant subset.
\end{proposition}

\begin{proof}
We need the following criterium that we borrow from \cite{HU}:

\begin{lemma}\label{l.uniqueminimal}
Let $h$ be a partially hyperbolic diffeomorphism. Suppose that for any two unstable leaves
$\cF^u(x_1)$ and $\cF^u(x_2)$ there are $n_1, n_2>0$ and a stable leaf $\cF^s(y)$ such that
$\cF^s(y) \cap h^{n_i}(\cF^u(x_i))\neq \phi$ for $i=1,2$.
Then, $h$ has a unique compact, invariant and $u$-saturated subset.
\end{lemma}

We will use the following terminology. Let $D_1$ be the graph of a continuous map $\phi_1: I^d\to I$,
where $I=[0,1]$. Given a point $x=(x_1,x_2)\in I^d\times I$, we say that
$x$ is \emph{above} $D_1$ if $x_2>\phi_1(x_1)$ and we say that 
$x$ is \emph{below} $D_1$ if $x_2>\phi_1(x_1)$.
Let $D_2$ be the image of a continuous injective map $\phi_2: I^d\to I^d\times I$,
we say $D_2$  \emph{crosses} $D_1$ if there are $y,z\in D_2$ such that $y$ is above $D_1$ and $z$ is below $D_1$.
It is easy to see that if $D_2$ crosses $D_1$ then $D_1\cap D_2\neq \emptyset$.

The next lemma is related to results in \cite{AV,D}:

\begin{lemma}\label{l.topologicaltransverse}
Let $f$ be a dynamically coherent accessible partially hyperbolic diffeomorphism with 1-dimensional center.
Then there are points $x_1\in M$ and $x_2\in\cF^s_1(x_1)$, where $\cF^s_1(x_1)$ denotes the ball contained in the leaf $\cF^s(x_1)$ with center $x_1$ and radius $1$,
such that the holonomy map $\cH^s_{x_2,x_1}: \cF^{cu}_{loc}(x_2)\to \cF^{cu}_{loc}(x_1)$ induced by the stable foliation satisfies
$$
\cH^s_{x_2,x_1}(\cF^u_1(x_2)) \text{ crosses }\cF^u_1(x_1).
$$
\end{lemma}

\begin{proof}
Suppose there are $x_1\in M$ and $x_2\in\cF^s_1(x_1)$ such that $\cH^s_{x_2,x_1}(\cF^u_1(x_2))$ is not contained in $\cF^u_1(x_1)$.
Then we may assume that there is $y\in \cH^s_{x_2,x_1}(\cF^u_1(x_2))$ which is below $\cF^u_1(x_1)$.
Take $x_3\in \cF^{cu}_{loc}(x_1)$ very close to $\cF^u_1(x_1)$ and still below $\cF^u_1(x_1)$. Then $x_1=\cH^s_{x_2,x_1}(x_2)$ is above $\cF^u_1(x_3)$.
On the other hand, since $\cF^u$ is continuous, $y$ is below $\cF^u_1(x_3)$.
Thus $\cH^s_{x_2,x_1}(\cF^u_1(x_2))$ crosses $\cF^u_1(x_3)$, as claimed in the lemma.

We are left to show that points $x_1$ and $x_2$ as in the previous paragraph do exist. Suppose otherwise.
Then for every $x\in M$ the union $\cup_{y\in \cF^u(x)}\cF^s(y)$ is a topological codimension-one submanifold $\cF^{su}(x)$ which is sub-foliated by $\cF^u$ and $\cF^s$.
The family $\cF^{su}(x)$, $x\in M$ is a topological foliation for which every leaf $\cF^{su}$ is an accessible class.
But it is clear that $\cF^{su}(x) \neq M$, and so this contradicts the assumption that $f$ is accessible.
This contradiction completes the proof.
\end{proof}

Since $\cF^u$ and $\cF^s$ vary continuously with the diffeomorphisms, crossing is a robust property.
Thus, Lemma~\ref{l.topologicaltransverse} has the following immediate consequence:

\begin{corollary}\label{c.robusttwist}
Let $f$ be an accessible partially hyperbolic diffeomorphism with 1-dimensional center.
Then there are a $C^1$-neighborhood $\cU$ of $f$, and two disjoint open sets $U_1, U_2$ such that, for any $g\in \cU$, and any pair of
points $x_1\in U_1$, $x_2\in U_2$, there is a stable leaf $\cF^s(y)$ of $g$ such that $\cF^s(y)$ intersects both $\cF^u_1(x_1)$ and $\cF^u_1(x_2)$.
\end{corollary}

We are ready to complete the proof of Proposition~\ref{p.uniqueminimalcomponent}.
Take a neighborhood $\cU$ of $f$ and open sets $U_1$ and $U_2$ as in Corollary~\ref{c.robusttwist}.
As already stated at the beginning of this section,  Proposition~\ref{p.nupositivesetinleaf} implies that  $\cup_{i>0}f^i(\cF^u_1(x))$ is a dense subset of $M$ for any $x\in M$.
Hence for each $x\in M$ and $i=1, 2$ there is $n_{x,i}\leq n_0$ such that $f^{n_{x,i}}(\cF^u_1(x))$ intersects $U_i$ for $i=1,2$.
Up to shrinking $\cU$ if necessary, this remains true for any $g\in\cU$.
Now it suffices to use the criterium in Lemma~\ref{l.uniqueminimal}.
\end{proof}

Let us proceed with the proof of Theorem~\ref{main.robust} in the rotation type case. Let $g$ denote a $C^2$ diffeomorphism $C^1$-close to $f$.
By Theorem~\ref{main.skeleton}, either $g$ has no hyperbolic periodic point, in which case it has a unique maximal measure;
or the ergodic maximal measures of $g$ have non-vanishing center exponents, and their supports are pairwise disjoint $u$-saturated invariant compact sets.
By Proposition~\ref{p.uniqueminimalcomponent}, there is at most one of such invariant set.
Hence, $g$ has a unique ergodic maximal measure $\mu_g^-$ with negative center exponent.
A similar argument shows that $g$ admits a unique ergodic maximal measure $\mu_g^+$ with positive center exponent.

Denote by $\Gamma^+$ (respectively, $\Gamma^-$) the map assigning to $g\in \cU\cap\Diff^2(M)$ its ergodic maximal measure with non-negative (respectively, non-positive) center exponent.

\begin{lemma}\label{l.uniquegibbs}
For $g\in \SPH1$ sufficiently close to $f$ in the $C^1$ topology, $\Gamma^-(g)$ is the unique $\nu$-Gibbs $u$-state of $g$.
\end{lemma}
\begin{proof}
If $g$ has no hyperbolic periodic orbit then by Theorem~\ref{main.skeleton}(a), it admits a unique ergodic maximal measure and it coincides with $\Gamma^-(g)$.
By Proposition~\ref{p.Gibbsustates}(a), $\Gamma^-(g)$ is also the unique $\nu$-Gibbs $u$-state for $g$, so we get the conclusion in this case.
If $g$ does have have some hyperbolic periodic orbit then, from Lemma~\ref{l.maximalmeasureandgibbsustates} and the previous observation that $g$ has a unique
ergodic maximal measure $\Gamma^-(g)=\mu^-_g$ with negative center exponent, we conclude that $g$ has a unique $\nu$-Gibbs $u$-state, and it coincides with $\Gamma^-(g)$.
Thus we get the conclusion also in this case.
\end{proof}

Consider a sequence of $C^2$ diffeomorphisms $\{g_n\}_{n=0}^\infty\subset \cU$ converging to $g$ in the $C^1$ topology.
According to Proposition~\ref{p.Gibbsustates}(c), any accumulation point of $\nu$-Gibbs $u$-states of $g_n$ is a $\nu$-Gibbs $u$-state of $g$.
Since we have just shown that the latter is unique, it follows that $\Gamma^-(g_n)$ converges to $\Gamma^-(g)$ when $n\to\infty$.
This proves that the map $\Gamma^{-}$ is continuous. The argument for $\Gamma^+$ is entirely analogous.

This proves Theorem~\ref{main.robust} in the rotation type case.

\subsection{Hyperbolic case}

Now suppose that $f$ has some hyperbolic periodic orbit. Then the same is true for any diffeomorphism in some neighborhood $\cU$ of $f$.
By Corollary~\ref{c.mostlycontracting}, every $g\in \cU\cap\Diff^2(M)$ has $\nu$-mostly contracting center.
Suppose that $f$ has $k\ge 1$ ergodic maximal measures $\{\mu^-_1,\cdots, \mu^-_k\}$ with negative center exponents.
By Lemma~\ref{l.maximalmeasureandgibbsustates}, these are precisely the ergodic $\nu$-Gibbs $u$-states of $f$.
By Theorem~\ref{main.skeleton}, the diffeomorphism $f$ has some skeleton $\cS(f)=\{p_1,\cdots,p_k\}$ such that the dimension of the stable manifold of each $p_i$ is equal to $\dim E^{cs}$.
By parts (c) and (d) of Proposition~\ref{p.skeleton}, the continuation $\{p_1(g),\cdots,p_k(g)\}$ is a pre-skeleton of $g$ and contains some skeleton of $g$.
In fact, it is itself a skeleton of $g$ if and only if no heteroclinic intersection was created between $p_i(g)$ and $p_j(g)$ for $1\leq i\neq j \leq k$ after perturbation.
Obviously, the number of elements of this (or any other) skeleton of $g$ is at most $k=\#\cS(f)$.
Using Theorem~\ref{main.skeleton} once more, we get that the number of ergodic maximal measures with negative center exponent of any $C^2$ diffeomorphism $g\in\cU$ is smaller than or
equal to the number of ergodic maximal measures with negative center exponent of $f$.

Consider any sequence $(f_n)_n$ of $C^2$ diffeomorphisms converging to $f$ in the $C^1$ topology,
and suppose $\{p_1(f_n),\cdots,p_k(f_n)\}$ is a skeleton of $f_n$ for each $n$.
By Theorem~\ref{main.skeleton}, each $f_n$ has exactly $k$ ergodic maximal measures $\mu^{n,-}_i$, $i=1,\cdots,k$ with negative center exponent.
Moreover, their supports are pairwise disjoint and, up to renumbering, we may assume that $p_i(f_n)\in\supp\mu^{n,-}_i$ for every $i$.
We want to prove that 
\begin{equation}\label{eq.convergence_i}
\left(\mu^{n,-}_i\right)_n \text{ converges to } \mu_i \text{ for every } i.
\end{equation}
Up to reordering, it is no restriction to consider $i=1$. Also, up to restricting to a subsequence, we may suppose that $\left(\mu^{n,-}_{1}\right)_n$ converges to some $\tilde{\mu}_{1}$.
Now, parts (b) and (c) of Proposition~\ref{p.Gibbsustates} imply that $\tilde{\mu}_{1}\in \Gibb^u_\nu(f)$ and can be written as a convex combination
$\tilde{\mu}_{1}=a_1\mu^-_1+\cdots+a_k\mu^-_k$. To prove \eqref{eq.convergence_i}, we just have to check that $a_j=0$ for every $j>1$.

Now, the next lemma asserts that otherwise $\{p_1(f_n),\cdots, p_k(f_n)\}$ is not a skeleton of $f_n$, which would contradict the currents assumptions.
So, to finish all we need is

\begin{lemma}
If $a_{j}>0$ for some $j>1$ then $\cF^u(p_1(f_n))$ has a transverse intersection with $W^s(p_j(f_n))$ for every large enough $n$.
\end{lemma}

\begin{proof}
Choose $B$ a small neighborhood of $p_j$ such that $\mu^-_j(B)=b>0$ and $W^s(p_j)$ has a transverse intersection with $\cF^u_{loc}(x)$ for every $x\in B$.
Take $n$ large enough that $\mu^{n,-}_1(B)>{a_j b}/{2}$ and $p_j(f_n)$ is close enough to $p_j(f)$ that its stable manifold $W^s(p_j(f_n))$ has a transverse
intersection with $\cF^u_n(x)$ for any $x\in B$.
In particular, $\supp(\mu^{n,-}_1)$ intersects $B$. By Proposition~\ref{p.Gibbsustates}, $\supp(\mu^n_1)$ is $u$-saturated, and by Theorem~\ref{main.skeleton},
we conclude that $\cF^u(\Orb(p_1(f_n)))$ intersects $B$. Hence, $\cF^u(\Orb(p_1(f_n))) \cap W^s(p_j(f_n))\neq \emptyset$.
However, this contradicts the assumption that $\{p_1(f_n),\cdots, p_k(f_n)\}$ is a skeleton and, consequently, there are no heteroclinic intersections between
its periodic orbits.
\end{proof}

This finishes the proof of Theorem~\ref{main.robust}.

\section{Proofs of Theorems~\ref{main.Nil} and ~\ref{main.Nilcontinuation}}

Let $M$ be a 3-dimensional nilmanifold different from $T^3$.
Recall that (Proposition~\ref{p.nilmanifold}), every partially hyperbolic diffeomorphism $f:M\to M$ is in $\SPH1$.

First we deduce Theorem~\ref{main.Nil}.
By Proposition~\ref{p.nilmanifold} and Theorem~\ref{main.skeleton}, there are two cases:
\begin{itemize}
\item either $f$ has no any hyperbolic periodic point, in which case it has a unique maximal measure, and is transitive;
\item or $f$ has some hyperbolic periodic point, and then both $\Gibb^u_\nu(f)$ and $\Gibb^s_\nu(f)$ contain each a unique element,
and so $f$ has exactly two ergodic maximal measures.
\end{itemize}
(In the second case uniqueness follows from the fact that $u$-saturated and $s$-saturated invariant compact sets are unique.)

The proof of Theorem~\ref{main.Nil} is complete. Now we prove Theorem~\ref{main.Nilcontinuation}.

As before, define $\Gamma^+(f)$ (respectively, $\Gamma^-(f)$) to be the ergodic maximal measure with non-negative (respectively, non-positive)
center exponent. To show that these maps are continuous at $f$ relative to the $C^1$ topology, we consider the following two situations:
\begin{itemize}
\item $f$ has no hyperbolic periodic orbits: then continuity is a direct corollary of Theorem~\ref{main.robust}(a);
\item $f$ has a hyperbolic periodic orbit: then every $C^2$ diffeomorphism $g$ in a $C^1$-neighborhood of $f$ also has a hyperbolic periodic orbit;
by Theorem~\ref{main.Nil}, $g$ admits a unique maximal measure with negative (positive) center exponent;
in particular, the number of ergodic maximal measures with negative (positive) center exponent is constant;
by Theorem~\ref{main.robust}(b), the ergodic maximal measure with negative (positive) center exponent varies continuously with the diffeomorphisms in the $C^1$ topology.
\end{itemize}
The proof is finished.

\end{document}